\documentclass[12pt,a4paper]{amsart}
\usepackage[utf8]{inputenc}

\usepackage{graphicx,color}
\usepackage{amsmath}
\usepackage{amssymb}
\usepackage{amsfonts}
\usepackage{amsthm}
\usepackage{amscd}
\usepackage{ae}
\usepackage{tikz}
\usepackage[T1]{fontenc}

\newcommand{\D}{\mathbb{D}}

\font\sets=msbm10 scaled \magstep1
\def\R{\text{\sets R}}

\def\C{\text{\sets C}}

\renewcommand\Re{\operatorname{Re}}
\renewcommand\Im{\operatorname{Im}}

\renewcommand{\bar}[1]{\overline{#1}}

\newcommand{\diam}{{\rm diam \,}}

\theoremstyle{plain}
\newtheorem{theorem}{Theorem}
\newtheorem{proposition}[theorem]{Proposition}

\theoremstyle{definition}

\author[O. Hirviniemi]{Olli Hirviniemi}
\address{University of Helsinki, Department of Mathematics and Statistics, P.O. Box 68, FIN-00014 University of Helsinki, Finland}
\email{olli.hirviniemi@helsinki.fi}

\author[L. Hitruhin]{Lauri Hitruhin}
\address{University of Helsinki, Department of Mathematics and Statistics, P.O. Box 68, FIN-00014 University of Helsinki, Finland}
\email{lauri.hitruhin@helsinki.fi}

\title[Pointwise contraction of mappings with integrable distortion]{Sharp pointwise contraction of mappings with integrable distortion which are quasiconformal in a disk}

\begin{document}

\begin{abstract}
We consider quasiconformal mappings of the unit disk that have a planar extension which have $p$-integrable distortion. In this paper, we establish a bound for the modulus of continuity for the inverse mapping and show sharpness of this bound. Furthermore, we also obtain bounds for the compression multifractal spectra of such mappings.
\end{abstract}

\maketitle
\section{Introduction}

A classical theorem of Morrey states that if $f$ is a $K$-quasiconformal mapping then it is $1/K$-H\"older continuous. The same is true for the inverse, so
\[
\frac{1}{C}|z-x|^{K} \leq |f(z)-f(x)| \leq C |z-x|^{1/K}.
\]
More recently the modulus of continuity bounds have been studied for more general families of mappings. In this paper we concentrate on homeomorphisms which are quasiconformal in a disk and have planar extension whose distortion is controlled. When the extension has $q$-exponentially integrable distortion the modulus of continuity was shown in \cite{Zap} to be
\begin{equation}\label{UppermodulusZapadinskay}
|f(z)-f(w)| \leq \frac{C}{\log^{q}\left( \frac{1}{|z-w|} \right)}.
\end{equation}
On the other hand, somewhat surprisingly, the corresponding lower bound
\begin{equation}\label{LowerModContyht}
	|f(z)-f(w)| \geq c|z-w|^{2K(1+\varepsilon)} \qquad \text{for any } z,w\in \overline{ \D},
	\end{equation} 
 obtained in \cite{HHPS} turns out to be close to the quasiconformal case.
Our first result establishes that a similar lower bound holds also for the case where the extension has only $p$-integrable distortion under a slight assumption on the image of the unit disk.

\begin{theorem}\label{PisteKontraktio}
Let $f: \mathbb{C} \to \mathbb{C}$ be a homeomorphism with $p$-integrable distortion, where $p \geq 1$,  which is $K$-quasiconformal in the unit disk. Furthemore, assume that the unit disk gets mapped to a domain $\Omega$ which is an image of a weak starlike domain $\Omega_L$ under a planar $L$-bi-lipschitz mapping, where $L\geq 1$. Then  
\begin{equation}\label{Kontraktio}
|f(z)-f(x)|\geq C_{L}|z-x|^{2\left(K+\frac{1}{p}\right)} \qquad \text{for any $z,x \in \overline{D}.$} 
\end{equation}
 Furthermore, this result is sharp in the sense that we can't replace the exponent at the right hand side of \eqref{Kontraktio} with a smaller one.
\end{theorem}
Note that the assumption on the domain $\Omega$ is quite general: It covers all cusp like domains, studied recently, for example, in \cite{GKT, KT, Xu}, and allows for spiraling and deformation of the weak star domain. Furthermore, we note that rather surprisingly the right hand side of \eqref{Kontraktio} is quite close, for any fixed $p\geq 1$, to the classical quasiconformal bound by Morrey instead of the pointwise bound 
$$|f(z)-f(x)|\geq e^{-c_{f,p}|z-x|^{-\frac{2}{p}}}$$
established in \cite{KT2} for planar mappings with $p$-integrable distortion. Hence we obtain a similar result as the bound \eqref{LowerModContyht} when the extension is assumed to have exponentially integrable distortion. However, we note that in the integrable distortion case the parameter $p$ plays a role in the result while in the exponentially integrable case the parameter $q$ does not affect the lower bound.

\medskip

\noindent The sharpness of Theorem \ref{PisteKontraktio} will be shown using mappings studied by Xu, see \cite{Xu}, which conformally map the unit disk to a cardioid type domain. These mappings produce almost sharp contraction when we restrict Theorem \ref{PisteKontraktio} to mappings which are conformal in the unit disk. For the  general  quasiconformal case we have to compose the maps constructed by Xu with the classical $K$-quasiconformal radial stretching map. 

\medskip

\noindent We also consider global behaviour in addition to the pointwise bound \eqref{Kontraktio}. Here the main tool are multifractal estimates, which capture the intuition that the extremal local behaviour cannot happen in a large set. For quasiconformal mappings it was shown in \cite{AIPS} that if for any $x \in E$ we have $|f(z)-f(x)| \sim |z-x|^\alpha$ for some $\alpha$, then 
\[
\dim E \leq 1 + \alpha - \frac{K+1}{K-1}|1-\alpha|,
\]
where negative dimension means that such behaviour cannot happen even on a pointwise scale. Similar results on the multifractal spectra for mappings with integrable or exponentially integrable distortion have been obtained in \cite{H4, H5}.
In the setting of Theorem \ref{PisteKontraktio} we obtain the following estimates.

\begin{theorem}\label{MFSPEKTRI}
Let $f: \mathbb{C} \to \mathbb{C}$ be a homeomorphism with $p$-integrable distortion, where $p \geq 1$,  which is $K$-quasiconformal in the unit disk. Furthemore, let the unit disk be mapped to a domain $\Omega$ which is an image of a weak starlike domain $\Omega_L$ under a planar $L$-bi-lipschitz mapping. Fix $0<s<2$ and assume that for any point $z\in A $ we can find a sequence $\lambda_n$, such that $|\lambda_n|\to 0$ when $n \to \infty$ and $z+\lambda_n \in \overline{D}$ for every $n$, for which
\begin{equation}\label{MFSyht}
|f(z)-f(z+\lambda_n)|\leq |\lambda_n|^{2K+ \frac{2-s+\epsilon}{p}}.
\end{equation}
Then we can bound the size of the set $A$ with $$\dim(A) \leq s.$$
\end{theorem}
Note that for any point $z\in D$ the local behaviour, that is, when $|\lambda_n| \to 0$, will be that of a $K$-quasiconformal mapping and hence the  lower bound for contraction  will get the form $c|\lambda_n|^K$. Hence Theorem \ref{MFSPEKTRI} is interesting when $0<s\leq 1$ and $A \subset \partial D$.

\section{Preliminaries}

Let $\Omega, \Omega' \subset \C$ be domains, and $f: \Omega \to \Omega'$ be an orientation-preserving homeomorphism. We say that $f$ is $K$-quasiconformal if $f \in W^{1,2}_{loc}(\Omega)$ and for almost every $z \in \Omega$ we have
\[
|Df(z)|^2 \leq K J_f(z).
\]
Here $|Df(z)| = \sup \{ |Df(z)v| : v \in \C, |v|=1\}$ and $J_f(z)$ is the Jacobian of $f$ at $z$.
More generally, $f$ is a mapping of finite distortion if

1) $f \in W^{1,1}_{loc}(\Omega)$

2) $Jf \in L^1_{loc}(\Omega)$

3) There is a measurable function $K_f$, finite almost everywhere, with $K_f(z) \geq 1$ such that $|Df(z)|^2 \leq K_f(z) J_f(z)$ for every $z \in \Omega$.

\medskip

\noindent From now on we denote the smallest such mapping by $K_f$ and call it the distortion.

\bigskip

\noindent However, in full generality the class of mappings of finite distortion has too wild behaviour for our purposes. They can, for example, map the whole real axis to a point or create cavities, see \cite{AIM} pages 528 and 530, which makes it impossible to have a meaningful theory about local stretching and contraction properties. Thus we restrict ourselves to study homeomorphic mappings of finite distortion, and furthermore assume that the distortion is $p$-integrable, where $p \geq 1$. That is, we assume 
\[
K_f \in L^p_{loc}(\Omega).
\]
Controlling the distortion is the classical way to restrict the class of mappings of finite distortion to a family for which a rich theory can be constructed. For us the crucial ingredient is that under this assumption we can prove the modulus inequality \eqref{ModInEq}.

\bigskip

\noindent A key component in our proofs is the modulus of path families, so we briefly recall the concept here. For a more detailed look at the topic, we refer to \cite{V}. Let $\Gamma$ be a collection of locally rectifiable paths. 
A non-negative measurable function $\rho: \C \to \R$ is said to be admissible with respect to the path family $\Gamma$ if
\[
\int_\gamma \rho(z) |dz| \geq 1
\]
for all paths $\gamma \in \Gamma$. The modulus of the path family $\Gamma$ is
\[
M(\Gamma) := \inf_\rho \int_\C \rho(z)^2 dA(z),
\]
where the infimum is taken over all admissible $\rho$. We define also weighted modulus for any weight $w: \C \to [0,\infty)$ (assumed to be measurable and locally integrable) by
\[
M_w(\Gamma) := \inf_\rho \int_\C \rho(z)^2 w(z) dA(z),
\]
where the infimum is again taken over all admissible $\rho$.
The important modulus inequality in our setting is the following. Let $f$ be a mapping of finite distortion, and suppose that the distortion $K_f$ is locally integrable. Then
\begin{equation}\label{ModInEq}
M(f(\Gamma)) \leq M_{K_f}(\Gamma).
\end{equation}

We say that a \textit{domain $\Omega$ is weak starlike} if for any two points $z,x \in \partial \Omega$ we can find a point $w\in \Omega$, which can depend on the choice of $z,x$, such that   $[w,z] \subset \overline{\Omega}$ and $[w,x] \subset \overline{\Omega}$. If there exists a single point $\omega \in \Omega$ which satisfies this condition for an arbitrarily chosen boundary points $z,x \in \partial \Omega$ then we say that the domain  $\Omega$ is starlike. Note that all the classical cusp domains, see, for example, \cite{GKT, KT,  Xu}, are starlike domains.

\bigskip

\noindent We will also briefly recall what Hausdorff dimension means. The $s$-dimensional Hausdorff measure is defined as
\[
\mathcal{H}^s(A) := \lim_{\delta \to 0+} \inf_{\mathrm{diam}(A_i) < \delta, A \subset \bigcup_i A_i} \sum_{i} \mathrm{diam}(A_i)^s.
\]
The set $A$ has the Hausdorff dimension $s$ if $\mathcal{H}^{s_1}(A) = \infty$ for any $s_1 < s$ and $\mathcal{H}^{s_2}(A) = 0$ for any $s_2 > s$.

\medskip

\noindent We need the following property of the Hausdorff dimension. Let $A$ be a set of Hausdorff dimension $s$. There is a constant $\delta_0 > 0$ with the following property. For any $s' < s$ and $0 < \delta < \delta_0$ we can find at least $\delta^{-s'}$ disjoint balls $B(z,\delta)$ with $z \in A$.


\section{Proof of Theorem \ref{PisteKontraktio}}
In this section we prove the pointwise contraction bound \eqref{Kontraktio} and verify that it is sharp.

\bigskip

\noindent Let $f: \mathbb{C} \to \mathbb{C}$ be a homeomorphism with $p$ integrable distortion that is $K$-quasiconformal in the unit disk. Assume first that both $z,x \in \partial D$ and, without loss of generality, that $x=1$ and set $|z-1|=r$. Denote $E=B\left(1, \frac{r}{3} \right) \cap \overline{D}$ and $F=B\left(z, \frac{r}{3} \right) \cap \overline{D}$ and assume 
\begin{equation}\label{KUVA1}
|f(z)-f(1) | \leq c|z-1|^a
\end{equation}
where $a>2K$. Our aim is to find an upper bound for $a$. 

\medskip

\noindent To this end we use the modulus inequality \eqref{ModInEq} with the path family $\Gamma$ connecting sets $E$ and $F$, which forces us to find good estimates  for the moduli $M_{K_f}(\Gamma)$ and $M(f(\Gamma))$.  

\medskip

\noindent For the modulus on the domain side let us denote by $g$ the planar $L$-bi-lipschitz map which maps the domain $\Omega$ to the weak star-like domain $\Omega_L$. Furthermore, let $w\in \Omega_L$ be such that the line segments $[w,g(f(1))]$ and $[w,g(f(z))]$ are contained in the closure $\overline{\Omega}_L$. The point $w$ exists due to the assumption that $\Omega_L$ is weak star. Let us then define $\bar{E}$ as the continuum of $$\bar{E}= [w,g(f(1))] \cap g(f(E)) $$  which contains the point $g(f(1))$. We define $\bar{F}$ in similar manner as the continuum of $$ \bar{F}=[w,g(f(z))] \cap g(f(F)) $$ containing the point $g(f(z))$. 

\medskip

\noindent Note that both $\bar{E}$ and $\bar{F}$ are line segments whose endpoints are mapped to the distance $\frac{r}{3}$ from the points $1$ and $z$ respectively under the mapping  $f^{-1} \circ g^{-1}.$

\medskip

\noindent Since $f$ restricted to the unit disk is a $K$-quasiconformal map we have a Beurling type estimate, see, for example, \cite{GKT} Lemma 3.6, 
\begin{equation*}
|a-b|\leq c \cdot \diam(A)^{\frac{1}{2K}},
\end{equation*}
where $A\subset f(\overline{D})$ is an arbitrary continuum with endpoints $f(a),f(b)$ and $a,b \in \overline{D}$.  This, together with the fact that mapping $g$ is  $L$-bi-lipschitz, ensures that the line segments $\bar{E}$ and $\bar{F}$ satisfy 
\begin{equation}\label{Pituus}
\diam\{ \bar{E} \}, \diam\{ \bar{F} \} \geq c r^{2K},
\end{equation}
where the constant $c$ depends on $L,f$. 

\medskip

\noindent Since mapping $g$ is a planar $L$-bi-lipschitz map it is also a $K_L$-quasiconformal. Thus, up to a constant which depends on $K_L$, we can estimate the modulus $M(f(\Gamma))$ with the modulus $M(g \circ f (\Gamma))$. Moreover, we can   estimate the modulus from below  by $M(\bar{\Gamma})$, where $\bar{\Gamma}$ is the family containing all paths connecting $\Bar{E}$ and $\Bar{F}$, due to the fact that $\bar{E} \subset g \circ f (E)$ and $\bar{F } \subset g \circ f (F)$. 

\medskip

\noindent Geometrically the line segments $\bar{E}$ and $\bar{F}$ are lying on a cone-like structure with the sides $[w,g(f(1))]$ and $[w, g(f(z))]$. Furthermore, since we assumed \eqref{KUVA1} and $g$ is $L$-bi-lipschitz, the width $|g(f(1))-g(f(z))| \approx |f(1)-f(z)|$ is much smaller than the length \eqref{Pituus} of $\bar{E}$ and $\bar{F}$. Thus we can estimate the modulus $M(\bar{\Gamma})$ from below by the modulus between top and bottom part of a rectangle with length $c_{L,f} \min \{\bar{E}, \bar{F} \}$ and height $c_{L,f}|f(z)-f(1)|$, which gives
\begin{equation}\label{KUVAMODULI}
M(f(\Gamma)) \geq c_L M(\bar{\Gamma}) \geq c_{L,f} \frac{r^{2K}}{|f(z)-f(1)|}.
\end{equation}
For the weighted modulus $M_{K_f}(\Gamma)$ we will estimate it from above using the path family $\tilde{\Gamma}$ which connects the set $E$ to the complement of the ball $B\left( 1, \frac{2r}{3} \right)$.  To estimate this further we use Hölder 
\begin{equation}\label{Moduli2}
\begin{split}
M_{K_f}(\Gamma) \leq M_{K_f}(\tilde{\Gamma}) & \leq \int_{B\left( 1, \frac{2r}{3} \right)} K_{f}(z) \rho_{0}^{2}(z) \; dA(z) \\&
\leq \left( \int_{B\left( 1, \frac{2r}{3} \right)} K_{f}^{p}(z) \; dA(z) \right)^{\frac{1}{p}}\left( \int_{\mathbb{C}} \rho_{0}^{\frac{2p}{p-1}} \; dA(z)  \right)^{\frac{p-1}{p}} \\ &
\leq c_{f,p,r} \left( \int_{\mathbb{C}} \rho_{0}^{\frac{2p}{p-1}} \; dA(z)  \right)^{\frac{p-1}{p}},
\end{split}
\end{equation}
where $c_{f,p,r} \to 0$ when $r\to 0$ as the distortion is $p$-integrable. Note that the integral part in the last line is  just the classical $\frac{2p}{p-1}$ modulus for the path family $\tilde{\Gamma}$ connecting boundaries of an annuli. Thus we can use the classical estimate, see, for example \cite{MRSY} remark 2.5, and get
\begin{equation*}
M_{K_f} \leq  c_{f,p,r}\left( \frac{1}{r} \right)^{\frac{2}{p}}.
\end{equation*}
Using the modulus inequality \eqref{ModInEq} together with the estimates \eqref{KUVAMODULI} and \eqref{Moduli2} for the moduli we obtain 
\begin{equation*}
c_{L,f} \frac{r^{2K}}{|f(z)-f(1)|} \leq c_{f,p,r}\left( \frac{1}{r} \right)^{\frac{2}{p}},
\end{equation*}
which gives the desired bound. Finally we note that the same proof holds when one or both of the points $z,x$ lie on the interior of the unit disk. 

\subsection{Sharpness of Theorem \ref{PisteKontraktio}}

Theorem \ref{PisteKontraktio} is sharp. We show this using the cardioid-type domains from \cite{Xu}.

\begin{proposition}
Let $K \geq 1$ and $p \geq 1$. For any $\varepsilon > 0$ there exists a starlike domain $\Omega$ and a mapping of $p$-integrable distortion $f: \C \to \C$ that is $K$-quasiconformal inside the unit disk such that for all $0<r<1$ we have
\[
\min_{z,x \in \overline{D}, |z-x|=r} |f(z)-f(x)| \leq Cr^{2\left(K+\frac{1}{p} - \varepsilon\right)}.
\]
\end{proposition}

\begin{proof}
Let $1 < s \leq 2$ such that $s=1+\frac{1}{p} - \varepsilon$. Define $\Delta_s$ to be the domain bounded by the curve
\[
\{(x,y) : -1 \leq x \leq 0, y = (-x)^s\},
\]
its reflection with respect to the real axis, and the unique image of a circle arc under the power mapping $z \mapsto z^2$ such that the boundary is continuously differentiable at $-1+i$ and $-1-i$. The shape of the domain $\Delta_s$ is illustrated by the Figure \ref{CuspDomain}. For a more rigorous definition of this domain, see \cite{Xu}.

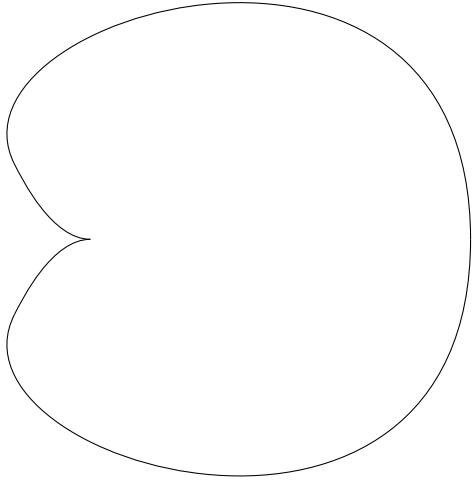
\begin{figure}
\begin{tikzpicture}
\draw (0,0) parabola (-1,1);
\draw (0,0) parabola (-1,-1);
\draw (-1,1) .. controls (-2,3) and (5,5) .. (5,0);
\draw (-1,-1) .. controls (-2,-3) and (5,-5) .. (5,0);
\end{tikzpicture}
\caption{Domain $\Delta_s$}.
\label{CuspDomain}
\end{figure}

Let $M_s$ be the image of $\Delta_s$ under the complex square root with the branch chosen to have $1^{1/2} = 1$, and take a conformal mapping $g_s: D \to M_s$ with $g_s(-1)=0$. It is shown in \cite{Xu} that the boundary of $M_s$ is Dini-smooth, and therefore $g_s$ extends to a bi-Lipschitz mapping from $\overline{D}$ to $\overline{M_s}$. This means that the conformal mapping $h_s: D \to \Delta_s$, $h_s(z) = g_s(z)^2$ satisfies $|h_s(x)| \leq C|x+1|^2$. Furthermore, it is shown on \cite{Xu} page 6 that the mapping $h_s$ satisfies 
\begin{equation}\label{Peilaus}
h_s(\bar{z})= \overline{h_s(z)}
\end{equation}
 for all $z\in D$.

Letting $x$ be a point on the unit circle close to $-1$ in the upper half-plane, we see that $$|h_s(x)-h_s(\overline{x})| = 2 \Im h_s(x) = 2 |\Re h_s(x)|^s \leq 2 |h_s(x)|^s \leq 2C|x+1|^{2s}.$$ By theorem 1.2 in \cite{Xu}, $h_s$ can be extended to a mapping of $p$-integrable distortion since $p < \frac{1}{s-1}$. Note that this family of mappings would be enough to show sharpness for Theorem \ref{PisteKontraktio} if we assume that $f$ is conformal inside the unit disk.

For the general case let $S_K: \C \to \C$ be the standard radial stretching:
\[
S_K(z) = \frac{z}{|z|}|z|^K.
\]
Take $\Omega = S_K(\Delta_s)$ and $f = S_K \circ h_s$. Note that $f$ is a homeomorphism since it is a composition of two homeomorphisms and that the condition \eqref{Peilaus} is satisfied for $S_k$, and thus also for $f$. Then for $x$ on the unit circle near $-1$ we have
\begin{equation*}
\begin{split}
|f(x)-f(\overline{x})| &= 2 \Im f(x) = 2 |\Re f(x)||\Re h_s(x)|^{s-1} \leq 2 |f(x)||h_s(x)|^{s-1}\\
& = 2 |h_s(x)|^{K+s-1} \leq 2C|x+1|^{2(K+s-1)},
\end{split}
\end{equation*} 
where the second equality comes from the similarity of the triangle with vertices at the origin, $\Re h_s(x)$ and  $\Im h_s(x)$ with the triangle with vertices at the origin, $\Re f(x)$ and $\Im f(x)$.
As $s=1+\frac{1}{p}-\varepsilon$ and $|x+1| \leq |x-\overline{x}|$ for $x$ on the unit circle close to $-1$, this finishes the proof.
\end{proof}

\section{Proof of Theorem \ref{MFSPEKTRI}}
Let us then expand the ideas from the pointwise case to study local compression in larger sets.  Assume that $f$ and set $A$ satisfy the assumptions of Theorem \ref{MFSPEKTRI}. Thus $A \subset D$ and for any point $z\in A$ we have a sequence $\lambda_n$ such that $|\lambda| \to 0$ and $z+\lambda_{z,n} \in \overline{D}$, which satisfies 
\begin{equation*}
|f(z)-f(z+\lambda_n)|\leq |\lambda_n|^{2K+ \frac{2-s+\epsilon}{p}}.
\end{equation*}
 Using Lemma 3.4 from \cite{H4} we note that there must be many radii $\lambda_{z,n_z}$ such that $$|\lambda_{z,n_n}| \in \left( \frac{1}{2^{m_i}}, \frac{1}{2^{m_i-1}} \right]=L_{m_i}$$ for some growing sequence $m_i$. To be more precise, there exists a sequence $m_i$ for which we can find of order $2^{m_i(\dim(A)-\epsilon)}$, where $\epsilon>0$ can be made arbitrary small, disjoint balls $B(z,5|\lambda_{z,n_z}|)$ such that $z\in A$ and $|\lambda_{z,n_z}| \in L_{m_i}$. 

\medskip

\noindent Our aim is to again use the modulus inequality \eqref{ModInEq} to obtain the bound for the dimension of the set $A$. Hence we need to find some suitable path family to study. To this end, fix $m_i$ for which we have many radii $\lambda_{z,z_n} \in L_{m_i}$ and define $$E= \bigcup_{|\lambda_{z,n_z}| \in L_{m_i}} B\left(z, \frac{|\lambda_{z,n_z}|}{3} \right)$$
where the points $z\in A$ are the centerpoints associated with the radius $\lambda_{z,r_n}$. Then define $$F=  \bigcup_{|\lambda_{z,n_z}| \in L_{m_i}} B\left(z +\lambda_{z,n_z}, \frac{|\lambda_{z,n_z}|}{3} \right) \cup \partial B\left(z,5|\lambda_{z,n_z}| \right).$$
Finally, define $\Gamma $ to be the path family connecting the sets $E$ and $F$. Due to the construction of these families, and the fact that the balls $B(z,5|\lambda_{z,n_z}|)$ are disjoint, this path family can be thought as the union of disjoint path families $\Gamma_i$ which connects one of the pairs $E_i$ and $F_i$ to each other. 

\medskip

\noindent To estimate the weighted modulus $M_{K_f}(\Gamma)$ we let $\rho_0$ be an arbitrary admissible function and use the Hölder inequality 
\begin{equation*}
\begin{split}
M_{K_f}(\Gamma)  & \leq \int_{B\left( 0, 2 \right)} K_{f}(z) \rho_{0}^{2}(z) \; dA(z) \\ &
\leq \left( \int_{B\left( 0, 2 \right)} K_{f}^{p}(z) \; dA(z) \right)^{\frac{1}{p}}\left( \int_{\mathbb{C}} \rho_{0}^{\frac{2p}{p-1}} \; dA(z)  \right)^{\frac{p-1}{p}} \\ &
\leq c_{f,p,r} \left( \int_{\mathbb{C}} \rho_{0}^{\frac{2p}{p-1}} \; dA(z)  \right)^{\frac{p-1}{p}}.
\end{split}
\end{equation*}
Here the last integral can be understood as the sum over the $\frac{2p}{p-1}$ moduli of the path families $\Gamma_i$, since function $\rho_0$ was an arbitrary admissible function. Let us denote for simplicity $R_{m_i}=\frac{1}{2^{m_i}}$, remember that $|\lambda_{z,n_z}| \in (R_{m_i}, 2R_{m_i}]$, and estimate each of the path families $\Gamma_i$ from above with the annuli $$B\left(z, \frac{2 |\lambda_{z,n_z}|}{3} \right) \setminus B\left(z, \frac{|\lambda_{z,n_z}|}{3} \right)$$
to obtain 
\begin{equation}\label{MODULI1MFSP}
\begin{split}
M_{K_f}(\Gamma)  & \leq c_{f,p,r} \left( \int_{\mathbb{C}} \rho_{0}^{\frac{2p}{p-1}} \; dA(z)  \right)^{\frac{p-1}{p}} \leq c_{f,p,r} \left( \frac{1}{R_{m_i}^{\frac{2}{p-1}+\dim(A)-\epsilon}}  \right)^{\frac{p-1}{p}} \\
& \leq \frac{c_{f,p,r}}{R_{m_i}^{\frac{2}{p}+\frac{\dim(A) (p-1)}{p}-\epsilon}}.
\end{split}
\end{equation}
For the modulus $M(f(\Gamma))$ on the image side we first divide it to a sum of the moduli $M(f(\Gamma_i))$ of which we can each one estimate from below using the exact same method as in the pointwise case. Hence, for each $F_i$ we obtain 
\begin{equation*}
M(f(\Gamma_i)) \geq C_{L,f} \frac{|\lambda_{z,n_z}|^{2K}}{|\lambda_{z,n_z}|^{2K + \frac{2-s+\epsilon}{p}}} \geq  \frac{C_{L,f}}{R_{m_i}^{\frac{2-s+\epsilon}{p}}},
\end{equation*}
 where we have used the assumption \eqref{MFSyht}. Thus, when summing over $i$, we obtain 
\begin{equation}\label{MODULI"MSFP}
M(f(\Gamma)) \geq c_{L,f,K,p} \frac{1}{R_{m_1}^{\frac{2-s+\epsilon}{p}+ \dim(A)-\epsilon}}.
\end{equation}
The estimates \eqref{MODULI1MFSP} and \eqref{MODULI"MSFP} for moduli together with the modulus inequality \eqref{ModInEq} give the desired estimate when $R_{m_i} \to 0$ and finish the proof of Theorem \eqref{MFSPEKTRI}.

\end{document}